\documentclass{amsart}
\usepackage{amsmath}
\usepackage{amssymb}
\usepackage{amsthm}
\usepackage{enumerate}
\usepackage[pdftex]{graphicx}
\usepackage{caption}
\theoremstyle{definition}
\newtheorem{definition}{Definition}[section]
\theoremstyle{plain}
\newtheorem{lemma}[definition]{Lemma}
\newtheorem{theorem}[definition]{Theorem}

\newtheorem{proposition}[definition]{Proposition}
\newtheorem{corollary}[definition]{Corollary}
\theoremstyle{remark}

\newtheorem{notation}[definition]{Notation}

\makeatletter
\@namedef{subjclassname@2020}{
  \textup{2020} Mathematics Subject Classification}
\makeatother

\newcommand{\mycl}{\operatorname{cl}}

\newcommand{\myrank}{\operatorname{rank}}
\newcommand{\mylen}{\operatorname{len}}
\newcommand{\myIso}{\operatorname{iso}}
\newcommand{\myLpt}{\operatorname{lpt}}

\begin{document}
\title[Definable Morse functions]{Morse functions definable in d-minimal structures}
\author[M. Fujita]{Masato Fujita}
\address{Department of Liberal Arts,
	Japan Coast Guard Academy,
	5-1 Wakaba-cho, Kure, Hiroshima 737-8512, Japan}
\email{fujita.masato.p34@kyoto-u.jp}

\begin{abstract}
	Fix a d-minimal expansion of an ordered field.
	We consider the space $\mathcal D^p(M)$ of definable $\mathcal C^p$ functions defined on a definable $\mathcal C^p$ submanifold $M$ equipped with definable $\mathcal C^p$ topology.
	The set of definable $\mathcal C^p$ Morse functions is dense in $\mathcal D^p(M)$.
\end{abstract}



\maketitle	

\section{Introduction}\label{sec:intro}

Several famous facts in differential geometry such as intersection theory and imbedding theorem of manifolds still hold even if we additionally require that sets and maps are definable in o-minimal expansion of an ordered field \cite{BO, D}.
Some of these hold under more relaxed conditions such as imbedding theorems \cite{FK}. 
Readers who are not familiar with o-minimal structures should consult \cite{D, PS, KPS}.
Morse functions on stratified sets satisfying Whitney conditions (a) and (b) are investigated in Morse theory on singular spaces. 
A classical textbook for Morse theory is \cite{Milnor} and Morse theory on singular spaces is discussed in \cite{GM}.
Loi showed that Morse functions on a stratified definable set form a dense and open subset of the space of definable $\mathcal C^p$ functions endowed with the Whitney topology in o-minimal structures \cite{Loi}. 
This is a definable and stratified version of the well-known result that the set of $\mathcal C^{\infty}$ Morse functions on a compact smooth manifold is a dense and open subset of the space of $\mathcal C^{\infty}$ functions.

We expect an assertion similar to Loi's result holds under more relaxed condition.
We consider d-minimal structures \cite{Miller_dmin, Fornasiero_dmin} in place of o-minimal structures.
An expansion of a dense linear order without endpoints $\mathcal F=(F,<,\ldots)$ is \textit{definably complete} if any definable subset $X$ of $F$ has the supremum and  infimum in $F \cup \{\pm \infty\}$ \cite{M}.
The structure $\mathcal F$  is \textit{d-minimal} if it is definably complete, and every definable subset $X$ of $F$ is the union of an open set and finitely many discrete sets, where the number of discrete sets does not depend on the parameters of definition of $X$ \cite{Fornasiero_dmin}.

Definable $\mathcal C^2$ functions not having degenerate critical points are called \textit{definable nondegenerate functions} for short.
A \textit{definable Morse function} is a definable nondegenerate function with distinct critical values in this paper. 
This paper studies Morse functions definable in d-minimal structures. 
Morse functions on singular spaces are studied in Loi's work, but we consider definable $\mathcal C^p$ submanifolds instead of singular spaces because Whitney stratification is not necessarily available in d-minimal structures \cite[Example 5.39]{Fornasiero_dmin} differently from o-minimal structures \cite{Loi2}.
Our main theorem asserts that the set of definable $\mathcal C^p$ Morse functions on a definable $\mathcal C^p$ submanifold $M$ is dense in the space $\mathcal D^p(M)$ of definable $\mathcal C^p$ functions on $M$ endowed with definable $\mathcal C^p$ topology defined in \cite{E}. 
%
In Section \ref{sec:preliminary}, we recall definitions and assertions in the previous studies which are necessary in this paper.
We prove the main theorem in Section \ref{sec:proof}.

We introduce the terms and notations used in this paper. 
Throughout, we fix a definably complete expansion of an ordered field $\mathcal F$ whose universe is $F$.
$p$ is a positive integer unless specified otherwise.
The term `definable' means `definable in the given structure with parameters.'
We assume that $F$ is equipped with the order topology induced from the linear order $<$ and the topology on $F^n$ is the product topology of the order topology.
An open box in $F^n$ is the product of $n$ many nonempty open intervals.
We denote  the closure of a set $A$ by $\mycl(A)$.
In the paper, $p$ is a positive integer unless another definition of $p$ is explicitly given.
We abbreviate the words `definable $\mathcal C^p$' by $\mathcal D^p$ in this paper.
For instance, a $\mathcal D^p$ function means a definable $\mathcal C^p$ function.

\section{Preliminaries}\label{sec:preliminary}

This paper often refers to $\mathcal D^p$ submanifolds.
We can find definitions of $\mathcal D^p$ submanifolds in many papers such as \cite[Definition 3.1]{Fuji}.
They are straightforward modifications of classical definitions of submanifolds. 
For readers' convenience, we recall a definition here and basic facts on them.
\begin{definition}\label{def:manifold}
	A definable subset $M$ of $F^n$ of dimension $d$ is called a \textit{$\mathcal D^p$ submanifold} of dimension $d$ if, for every point $x$ in $M$, there exists a $\mathcal D^p$ diffeomorphism $\varphi:U \to V$ from a definable open neighborhood $U$ of $x$ in $F^n$ onto a definable open neighborhood $V$ of the origin in $F^n$ such that $\varphi(x)$ is the origin and the equality $\varphi(M \cap U)= V \cap (F^d \times \{\overline{0}_{n-d}\}) $ holds.
	Here, $\overline{0}_{n-d}$ denotes the origin in $F^{n-d}$.
	A definable map $f:M \to F$ defined on a $\mathcal D^p$ submanifold $M$ is \textit{of class $\mathcal C^p$} if the composition $f \circ (\varphi|_{X \cap U})^{-1}$ is of class $\mathcal C^p$. 
	We define $\mathcal D^p$ map between $\mathcal D^p$ submanifolds similarly.
	We can define the tangent bundle $\pi:TM \to M$ of a $\mathcal D^p$ submanifold $M$ in the classical way.
	It is known that the $TM$ is a $\mathcal D^{p-1}$ submanifold.
	See \cite[Section 2]{BO} for instance.
	The structures are assumed to be o-minimal in \cite{BO}, but assertions in  \cite[Section 2]{BO} hold for all definably complete expansions of ordered fields.
	A \textit{$\mathcal D^{p-1}$ vector field} is a $\mathcal D^{p-1}$ map $D:M \to TM$ with $\pi \circ D = \operatorname{id}$.
	
	For a given point $x$ in a $\mathcal D^p$ submanifold $M$ of dimension $d$, we can always find a definable open subset $U$ of $M$ containing the point $x$, a definable open subset $V$ of $F^d$ containing the origin and a $\mathcal D^p$ diffeomorphism $\varphi: U \to V$ such that $\varphi(x)=\overline{0}_d$.
	The pair $(U,\varphi)$ is called a \textit{coordinate neighborhood} of $M$ around $x$.
%
\end{definition}

We have introduced basic notions.
We begin to introduce lemmas which hold in definably complete expansions of ordered fields.
The following preimage theorem holds:
\begin{lemma}\label{lem:local_submersion}
	Let $M$ and $N$ be $\mathcal D^p$ submanifolds and $f:M \to N$ be a definable submersion; that is, $d_xf:T_xM \to T_{f(x)}N$ is surjective for every $x \in M$.
	Let $Y$ be a $\mathcal D^p$ submanifold of $N$.
	%
	Then $X:=f^{-1}(Y)$ is a definable submanifold of $M$ of dimension $=(\dim Y+\dim N-\dim M)$.
	The tangent space $T_xX$ of $X$ at $x \in X$ is the inverse image $df_x^{-1}(T_{f(x)}Y)$ of $T_{f(x)}Y$ under the linear map $d_xf:T_xM \to T_{f(x)}N$.
\end{lemma}
\begin{proof}
	This definable version of a classical result is proven in the same manner as the original one using a definable version of implicit function theorem obtained from definable inverse function theorem in \cite[Theorem 2.1]{Fuji}.
	We omit the details of this proof.
\end{proof}


We next define dimension of a definable set.
\begin{definition}\label{def:dim}
	We consider that $F^0$ is a singleton with the trivial topology.
	Let $X$ be a nonempty definable subset of $F^n$.
	The dimension of $X$ is the maximal nonnegative integer $d$ such that $\pi(X)$ has a nonempty interior for some coordinate projection $\pi:F^n \rightarrow F^d$.
	We set $\dim(X)=-\infty$ when $X$ is an empty set.
\end{definition}
We introduced two different notion of dimensions for $\mathcal D^p$ submanifolds in Definition \ref{def:manifold} and Definition \ref{def:dim}.
These two concepts coincide in d-minimal structures \cite[Lemma 3.4]{Fuji}.
We collect basic properties of d-minimal structures from the previous studies.
\begin{proposition}\label{prop:others}
	Suppose $\mathcal F$ is d-minimal.
	The following assertions hold:
	\begin{enumerate}
		\item[(1)] Every definable closed set is the zero set of a $\mathcal D^p$ function.
		\item[(2)] Let $f:U \to F$ be a definable function.
		There exists a definable closed subset $D$ of $U$ such that $D$ has an empty interior and the restriction of $f$ to $U \setminus D$ is of class $\mathcal C^p$.
		\item[(3)] Let $f:A \to F^n$ be a definable map. Then the inequality $\dim f(A) \leq \dim A$ holds.
		\item[(4)] The equality $\dim (A \cup B)=\max\{\dim A, \dim B\}$ holds for definable sets $A$ and $B$.
		\item[(5)] The equality $\dim \mycl(A)=\dim A$ holds for a definable set $A$.
		\item[(6)] Let $f :X \to Y$ be a definable surjective map such that $\dim f^{-1}(y)$ is independent of the choice of $y \in Y$.
		We have $\dim X = \dim Y+\dim f^{-1}(y)$.
		\item[(7)] Let $\pi:F^{m+n} \rightarrow F^m$ be a coordinate projection.
		Let $X$ and $Y$ be definable subsets of $F^m$ and $F^{m+n}$, respectively,  satisfying the equality $\pi(Y)=X$.
		There exists a definable map $\varphi:X \rightarrow Y$ such that $\pi(\varphi(x))=x$ for all $x \in X$.
		\item[(8)] Let $A$ be a definable locally closed subset of $F^n$.
		Let $g:A \to F$ be a definable continuous map.
		Let $\mathcal H$ be a finite set of definable locally bounded functions $A \setminus g^{-1}(0) \to F$.
		Then there exists a definable $\mathcal C^p$ function $\phi:F \to F$ satisfying the following conditions:
		\begin{itemize}
			\item $\phi$ is odd, strictly increasing, surjective and $p$-flat at $0$;
			\item $\phi(1)=1$;
			\item $\lim_{x \to y}\phi(g(x))h(x)=0$ for each $y \in g^{-1}(0)$ and $h \in \mathcal H$.
		\end{itemize}
	\end{enumerate}
\end{proposition}
\begin{proof}
	(1) See \cite[2.9]{MT_zero}.
	It is originally proven when $F=\mathbb R$.
	We can prove it in the same manner as the original for general $F$.
	(2) \cite[Lemma 3.14]{Fornasiero_dmin}.
	(3) \cite[Lemma 4.6]{Fornasiero_dmin}.
	(4) \cite[Lemma 4.5]{Fornasiero_dmin}.
	(5)  \cite[Theorem 3.10]{Fornasiero_dmin}.
	(6) \cite[Lemma 4.5]{Fornasiero_dmin} and \cite[Corollary 1.5(iii)]{D3}.
	(7) \cite{Miller-choice}.
	(8) \cite[1.5]{MT_zero}.
\end{proof}

\begin{corollary}\label{cor:submanifold}
	Every $\mathcal D^p$ submanifold of $F^n$ is $\mathcal D^p$ diffeomorphic to a closed $\mathcal D^p$ submanifold of $F^{n+1}$. 
\end{corollary}
\begin{proof}
	We can easily construct a $\mathcal D^p$ diffeomorphism using Proposition \ref{prop:others}(1).
\end{proof}

\begin{corollary}\label{cor:nondegenerate}
	The set of critical values of a given definable nondegenerate function is of dimension zero.  
\end{corollary}
\begin{proof}
	The set $S$ of critical points is discrete.
	We can prove $\dim S=0$ by using Proposition \ref{prop:others}(7),(2).
	We omit the details.
	The corollary follows from this fact and Proposition \ref{prop:others}(3).
\end{proof}

\section{Proof of main theorem}\label{sec:proof}

We begin to prove our main theorem.
We use the following technical definition.

\begin{definition}\label{def:multi_graph}
	Let $\pi:F^n \to F^d$ be a coordinate projection.
	A $\mathcal D^p$ submanifold $M$ of $F^n$ of dimension $d$ is called a \textit{$\mathcal D^p$ multi-valued graph (with respect to $\pi$)} if, for any $x \in M$, there exist an open box $U$ in $F^n$ containing the point $x$ and a $\mathcal D^p$ map $\tau:\pi(U) \to F^n$ such that $M \cap U=\tau(\pi(U))$ and $\pi \circ \tau$ is the identity map defined on $\pi(U)$. 
\end{definition}

The following lemma claims that a $\mathcal D^p$ submanifold is covered by finitely many $\mathcal D^p$ multi-graphs.

\begin{lemma}\label{lem:multi_graph_cover}
	Let $M$ be a $\mathcal D^p$ submanifold $M$ of dimension $d$.
	Let $\Pi_{n,d}$ be the set of coordinate projections from $F^n$ onto $F^d$.
	Let $U_{\pi}$ be the set of points $x$ at which there exist an open box $U$ in $F^n$ containing the point $x$ and a $\mathcal D^p$ map $\tau:\pi(U) \to F^n$ such that $M \cap U=\tau(\pi(U))$ and $\pi \circ \tau$ is the identity map defined on $\pi(U)$. 
	Then $U_{\pi}$ is a $\mathcal D^p$ multi-valued graph with respect to $\pi$ and 
	$\{U_{\pi}\}_{\pi \in \Pi_{n,d}}$ is a definable open cover of $M$.
\end{lemma}
\begin{proof}
	It is obvious that $U_{\pi}$ is a $\mathcal D^r$ multi-valued graph with respect to $\pi$.
	It is also obvious that $U_{\pi}$ is open in $M$.
	The family $\{U_{\pi}\}_{\pi \in \Pi_{n,d}}$ is a definable open cover of $M$ by \cite[Lemma 3.5, Corollary 3.8]{Fuji}.
\end{proof}

We use the following notation many times.

\begin{notation}\label{not:multi_graph}
	Let $\pi:F^n \to F^d$ be the projection coordinates onto the first $d$ coordinates.
	Let us consider a $\mathcal D^p$ multi-valued graph $M$ with respect to $\pi$.
	Let $f:M \to F$ be a $\mathcal D^p$ function.
	Fix $1 \leq i \leq d$.
	We define a $\mathcal D^{p-1}$ function $g:M \to F$ as follows:
	We can find an open box $U$ containing the point $x$ and $\tau:\pi(U)  \to F^n$ as in Definition \ref{def:multi_graph}.
	The restriction of $g$ to $U \cap M$ is defined by $$g(x)=\dfrac{\partial (f \circ \tau)}{\partial x_i}(\pi(x)).$$
	We can easily verify that the map $g$ is well-defined.
	We denote $g$ by $\dfrac{\partial f}{\partial x_i}$ in this paper by abuse of notations.
	We put $df=\left(\dfrac{\partial f}{\partial x_1},\ldots,\dfrac{\partial f}{\partial x_d}\right)$.
	Observe that $x \in M$ is a critical point of $f$ if and only if $df(x)=0$.
\end{notation}

We use the following result without notice in this paper:
\begin{lemma}
	Let $\pi:F^n \to F^d$ be the projection coordinates onto the first $d$ coordinates.
	Let us consider a $\mathcal D^p$ multi-valued graph $M$ with respect to $\pi$.
	Let $f:M \to F$ be a $\mathcal D^p$ function.
	The map $\dfrac{\partial f}{\partial x_k}$ is definable and of class $\mathcal C^{p-1}$ for every $1 \leq k \leq d$.
\end{lemma}
\begin{proof}
	It is obvious that $\dfrac{\partial f}{\partial x_k}$ is of class $\mathcal C^{p-1}$.
	We show that $\dfrac{\partial f}{\partial x_k}$ is definable.
	For every $x \in F^n$ and $t_1,\ldots, t_n>0$, we set $\mathcal B(x,t_1,\ldots,t_n)=\{y=(y_1,\ldots, y_n)\in F^n\;|\; |x_i-y_i|<t_i \text{ for }1 \leq i \leq n\}$, where $x=(x_1,\ldots,x _n)$.
	Consider the set 
	\begin{align*}
		X=&\{(x,t_1,\ldots,t_n) \in M \times F^n\;|\; t_i>0 \text{ for every } 1 \leq i \leq n
		\text{ and the set } \\ &M \cap \mathcal B(x,t_1,\ldots,t_n) \text{ is the graph of a } \mathcal D^p \text{ map defined on } \pi(\mathcal B(x,t_1,\ldots,t_n))\}.
	\end{align*}
	Let $\Pi_1:X \to M$ and $\Pi_2:X \to F^n$ be the restrictions of the projections of $M \times F^n$ onto $M$ and $F^n$, respectively.
	The map $\Pi_1$ is surjective by the definition of multi-graphs.
	We can find a definable map $\tau:M \to X$ such that $\tau \circ \Pi_1$ is the identity map on $M$ by Proposition \ref{prop:others}(7).
	Let $\rho_i:F^n \to F$ be the projection onto the $i$-th coordinate and define the definable map $\widetilde{t_i}:M \to F$ by $\widetilde{t_i}(x)=\rho_i(\Pi_2(\tau(x)))$.
	
	We next consider the set $$Y=\bigcup_{x \in M}\{x\} \times (M \cap \mathcal B(x,\widetilde{t_1}(x), \ldots, \widetilde{t_n}(x))) \subseteq M \times F^n$$
	and $Y \langle x \rangle =\{y \in F^n\;|\; (x,y) \in Y\}=   M \cap \mathcal B(x,\widetilde{t_1}(x), \ldots, \widetilde{t_n}(x))$ for every $x \in M$.
	By the definition, $Y \langle x \rangle$ is the graph of a $\mathcal D^r$ map defined on $\pi(Y \langle x \rangle)=\pi(\mathcal B(x,\widetilde{t_1}(x), \ldots, \widetilde{t_n}(x)))$ for every $x \in M$.
	Set $$Z=\{(x,z_1,\ldots,z_d) \in M \times F^d\;|\; (z_1,\ldots, z_d) \in \pi(B(x,\widetilde{t_1}(x), \ldots, \widetilde{t_n}(x)))\}.$$
	We can uniquely find $y_{d+1}(x,z_1,\ldots,z_d),  \ldots, y_{n}(x,z_1,\ldots,z_d)$ so that $$(z_1,\ldots,z_d,y_{d+1}(x,z_1,\ldots,z_d),  \ldots, y_{n}(x,z_1,\ldots,z_d)) \in Y \langle x \rangle$$ for every $x \in M$ and $z_1,\ldots, z_d \in \pi(Y\langle x \rangle)$
	because $Y\langle x \rangle$ is a graph of definable function defined on an open box $\pi(Y\langle x \rangle)$.
	The above procedure gives definable functions $y_{d+1}, \ldots y_n:Z \to F$ which are of class $\mathcal C^p$ with respect to the coordinates $z_1,\ldots, z_d$.
	Let $h:Z \to F$ be the definable function given by $h(x,z)=f(z,y_1(x,z),\ldots, y_n(x,z))$.
	It is obvious that $\dfrac{\partial f}{\partial x_k}(x)=\dfrac{\partial h}{\partial z_k}(x,x_1,\ldots,x_d).$
	This implies that $\dfrac{\partial f}{\partial x_k}$ is definable.
\end{proof}

\begin{proposition}[Definable Sard's theorem]\label{prop:definable_Sard}
	Suppose that $\mathcal F$ is d-minimal.
	Let $M$ be a $\mathcal D^1$ submanifold of $F^m$ of dimension $d$ and $f=(f_1,\ldots, f_n):M \to F^n$ be a $\mathcal D^1$ map.
	The set of critical values of $f$ is definable and of dimension smaller than $n$. 
\end{proposition}
\begin{proof}
	The proposition is trivial when $d<n$ by Proposition \ref{prop:others}(3).
	We consider the case in which $d \geq n$ in the rest of proof. 
	
	We denote the set of critical values of $f$ by $\Sigma_f$.
	The $\mathcal D^1$ manifold $M$ is covered by finitely many $\mathcal D^1$ multi-valued graphs $U_1,\ldots, U_k$ by Lemma \ref{lem:multi_graph_cover}.
	The equality $\Sigma_f=\bigcup_{i=1}^k \Sigma_{f|_{U_i}}$ obviously holds, where $f|_{U_I}$ is the restriction of $f$ to $U_i$.
	The set $\Sigma_f$ is definable if $ \Sigma_{f|_{U_i}}$ is definable for every $1 \leq i\leq k$.
	We have $\dim \Sigma_f=\max \{\dim \Sigma_{f|_{U_i}}\;|\;1 \leq i \leq k\}$ by Proposition \ref{prop:others}(4).
	Therefore, we may assume that $M$ is a $\mathcal D^1$ multi-valued graph with respect to a coordinate projection $\pi$.
	We may further assume that $\pi$ is the projection onto the first $d$ coordinates by permuting the coordinates if necessary.
	
	Let $Jf:M \to F^{d \times n}$ be the map giving the Jacobian matrix $(\frac{\partial f_i}{\partial x_j})_{1 \leq i \leq n, 1 \leq j \leq d}$.
	Here, we used the notations introduced in Notation \ref{not:multi_graph}.
	Set $$\Gamma_f=\{x \in M\;|\; \operatorname{rank}(Jf(x)) < n\}.$$ 
	The set $\Gamma_f$ is definable.
	The set $\Sigma_f$ is also definable because the equality $\Sigma_f=f(\Gamma_f)$ holds.
	
	Assume for contradiction that $\dim \Sigma_f=n$.
	The definable set $\Sigma_f$ contains a nonempty open box $U$.
	We can take a definable map $g:U \to \Sigma_f$ such that $f \circ g$ is an identity map on $U$ by Proposition \ref{prop:others}(7).
	We may assume that $g$ is of class $\mathcal C^1$ by Proposition \ref{prop:others}(2) by shrinking $U$ if necessary.
	By differentiation, the matrix $Jf(g(x)) \cdot Jg(x)$ is the identity matrix of size $n$.
	It implies that $Jf(g(x))$ has rank at least $n$, which contradicts the definition of $\Gamma_f$. 
\end{proof}

\begin{corollary}\label{cor:Sard}
	Suppose $\mathcal F$ is d-minimal.
	Let $M$, $N$ and $P$ be $\mathcal D^1$ submanifolds and $f:M \times P\to N$ be a $\mathcal D^1$ map.
	If $y \in N$ is a regular value of $f$, there exists a definable subset $D$ of $P$ such that $\dim D<\dim P$ and $y$ is a regular value of $f_p$ given by $M \ni x \mapsto f_p(x):=f(x,p) \in F$.
\end{corollary}
\begin{proof}
	The corollary is obvious when $y \notin f(M \times P)$.
	We assume that $y \in f(M \times P)$.
	Since $y$ is a regular value of $f$, $f^{-1}(y)$ is a $\mathcal D^1$ submanifold of $M \times P$ by Lemma \ref{lem:local_submersion}.
	Let $\pi$ be the restriction of the projection $M \times P \to P$ to $f^{-1}(y)$.
	Let $D$ be the set of critical values of $\pi$.
	We have $\dim D<\dim P$ by Proposition \ref{prop:definable_Sard}.
	
	We show that $y$ is a regular value of $f_p$ when $p \notin D$.
	Fix $p \in P \setminus D$.
	It is obvious when $y \notin f_p(M)$.
	We take $x \in M$ such that $f_p(x)=y$.
	Since $y$ is a regular value of $f$, the map $df_{(p,x)}:T_pP \times T_xM=T_{(p,x)}(P \times M) \to T_yN$ is surjective.
	Let $w$ be an arbitrary element of $T_y(N)$.
	We can find $(u,v) \in T_pP \times T_xM$ such that $df_{(p,x)}(u,v)=w$.
	Since $y$ is a regular value of $f_p$, the map $d\pi_{(p,x)};T_{(p,x)}f^{-1}(y) \to T_pP$ is surjective.
	We can take $v' \in T_xM$ so that $(u,v') \in T_{(p,x)}f^{-1}(y)$.
	We have $df_{(p,x)}(u,v')=0$ because $f$ is constant on $f^{-1}(p)$.
	We have $d(f_p)_x(v-v')=df_{(p,x)}(0,v-v')=df_{(p,x)}((u,v)-(u',v))=df_{(p,x)}(u,v)-df_{(p,x)}(u,v')=w$.
	It means that $y$ is a regular value of $f_p$.
\end{proof}

\begin{lemma}\label{lem:Morse_basic2}
	Suppose that $\mathcal F$ is d-minimal.
	Let $\pi$ be the coordinate projection of $F^n$ onto the first $d$ coordinates.
	Let $U$ be a $\mathcal D^2$ multi-valued graph with respect to $\pi$ and $f:F^k \times U \to F$ be a $\mathcal D^2$ function.
	The function $f_t:U \to F$ is given by $f_t(x):=f(t,x)$ for each $t \in F^k$.
	Assume that the definable map $\Phi:F^k \times U \to F^d \times F^d$ given by $\Phi(t,x)=(df_t(x),\pi(x))$ is a submersion, where $df_t$ is the map defined in Notation \ref{not:multi_graph}.
	Set $$Z=\{t \in F^k\;|\; f_t \text{ has a degenerate critical point}\}.$$
	Then the inequality $\dim Z<k$ holds.
\end{lemma}
\begin{proof}
	Let $\Pi:F^k \times U \to F^k$ be the canonical projection.
	Consider the zero section $U_0=\{0\} \times F^d$ in $F^d \times F^d$.
	We fix an arbitrary point $u_0 \in U$.
	If we take a sufficiently small open box $B$ containing the point $x_0$, the restriction $\pi|_{B \cap U}$ of $\pi$ to $B \cap U$ is a $\mathcal D^2$ diffeomorphism onto $\pi(B)$.
	Let $\varphi=(\varphi_1,\ldots, \varphi_n):\pi(B) \to F^n$ be the inverse of $\pi|_{B \cap U}$.
	Observe that $\varphi_i$ is the coordinate projection onto the $i$-th coordinate for each $1 \leq i \leq d$.
	
	The inverse image $\Phi^{-1}(U_0)$ is a $\mathcal D^1$ manifold of $F^k \times U$ of  dimension $k$ by Lemma \ref{lem:local_submersion}.
	In addition, the tangent space $T_{(t_0,x_0)}\Phi^{-1}(U_0)$ is the inverse image of $\{0\} \times F^d$ under the differential $d\Phi_{(t_0,x_0)}:F^k \times T_{x_0}U \to F^d \times F^d$ under the identifications of $T_{t_0}F^k$ and $T_tF^d$ with $F^k$ and $F^d$, respectively.
	We want to show that $t \in Z$ if and only if $t$ is a critical value of the restriction of $\Pi$ to  $\Phi^{-1}(U_0)$.
	If this claim holds, we have $\dim Z<k$ by Proposition \ref{prop:definable_Sard}.
	
	Observe that $ d(\text{id} \times \varphi)_{(t_0,\pi(x_0))}:F^k \times F^d \to F^k \times T_{x_0}U$ is a linear isomorphism because $\text{id} \times \varphi$ is a $\mathcal D^2$ diffeomorphism.
	The differential $ d(\Phi \circ (\text{id} \times \varphi))_{(t_0,\pi(x_0))}=d\Phi_{(t_0,x_0)} \cdot d(\text{id} \times \varphi)_{(t_0,\pi(x_0))} $ is given by the matrix of the form
	\begin{equation*}
		\left(
		\begin{array}{cc}
			* & H(t_0,x_0)\\
			0_{d,k} & I_d
		\end{array}
		\right),
	\end{equation*}
	where $0_{d,k}$ denotes the zero matrix of size $d \times k$ and $I_d$ is the identity matrix of size $d \times d$.
	Here, $H(t_0,x_0)$ denotes the matrix of size $d \times d$ given by
	$\left(
	\dfrac{\partial^2 f_{t_0}}{\partial x_i \partial x_j}(x_0)
	\right)_{1 \leq i,j \leq d}$
	under the notation in Notation \ref{not:multi_graph}.
	The inverse image $$L:=(d(\text{id} \times \varphi)_{(t_0,\pi(x_0))})^{-1} (T_{(t_0,x_0)}\Phi^{-1}(U_0))$$ is the kernel of the matrix 
	\begin{equation*}
		\left(
		\begin{array}{cc}
			* & H(t_0,x_0)
		\end{array}
		\right)
	\end{equation*}
	and it is of dimension $k$.
	Let $\Pi'$ be the restriction of $\Pi$ to $\Phi^{-1}(U_0)$.
	Set $V=(\text{id} \times \varphi)^{-1}(\Phi^{-1}(U_0))$.
	We can identify $d(\Pi' \circ (\text{id} \times \varphi)|_V)_{(t_0,\pi(x_0))}=d\Pi'_{(t_0,x_0)} \cdot d(\text{id} \times \varphi)_{(t_0,\pi(x_0))}|_L$ with the restriction $q:L \to F^k$ of the projection $F^k \times F^d \to F^k$ to $L$ under the natural identification of $T_{x_0}\pi(B)$ with $F^d$.
	Therefore, $(t_0,x_0)$ is a critical point of $\Pi'$ if and only if $H(t_0,x_0)$ is not invertible.
	This implies that $t_0 \in Z$ if and only if $t_0$ is a critical value of $\Pi'$.
\end{proof}

\begin{lemma}\label{lem:Morse_basic}
	Suppose that $\mathcal F$ is d-minimal.
	Let $\pi$ be the coordinate projection of $F^n$ onto the first $d$ coordinates.
	Let $U$ be a $\mathcal D^2$ multi-valued graph with respect to $\pi$ and $f,\varphi:U \to F$ be $\mathcal D^2$ functions.
	Assume further that $\varphi$ is nonzero everywhere.
	Let $D$ be a definable subset of $F$ of dimension $\leq 0$.
	We can find $a_0,\ldots, a_d \in F$ such that the definable function $\Phi:U \to F$ given by  $\Phi(x_1,\ldots,x_n) =f(x_1,\ldots,x_n)+a_0\varphi(x)+\sum_{i=1}^d a_ix_i\varphi(x)$ is a Morse function, every element in $D$ is a regular value of $\Phi$ and $|a_0|,\ldots,|a_d|$ are sufficiently small.
\end{lemma}
\begin{proof}
	Consider the map $\Psi:F^{d+1}\times U \to F$ defined by $\Psi(t,x)=f(x)+t_0\varphi(x)+\sum_{i=1}^d t_ix_i\varphi(x)$.
	Let $\Psi_t:U \to F$ be the function defined by $\Psi_t(x)=\Psi(t,x)$ for each $t \in F^{n+1}$.
	We have $d\Psi_t=(g_1(x,t),\ldots, g_d(x,t)),$ where $$g_j(x,t)=\dfrac{\partial f}{\partial x_j}+t_0\dfrac{\partial \varphi}{\partial x_j}+t_j\varphi(x)+\sum_{k=1}^dt_kx_k\dfrac{\partial \varphi}{\partial x_j}$$
	under the notation in Notation \ref{not:multi_graph}.
	We first verify that the definable map $\widetilde{\Psi}:(t,x) \mapsto (d\Psi_t(x),\pi(x))$ is a submersion. 
	The Jacobian $J\widetilde{\Psi}$ of $\widetilde{\Psi}$ is given by
	\begin{equation*}
		\left(
		\begin{array}{cccccc}
			\frac{\partial\varphi}{\partial x_1} & \varphi+x_1\frac{\partial\varphi}{\partial x_1} & x_2\frac{\partial\varphi}{\partial x_1} & \cdots & x_n\frac{\partial\varphi}{\partial x_1} & * \\
			\frac{\partial\varphi}{\partial x_2} & x_1 \frac{\partial\varphi}{\partial x_2} & \varphi +x_2\frac{\partial\varphi}{\partial x_2} & \cdots & x_n\frac{\partial\varphi}{\partial x_2}  & *\\
			\vdots & \vdots & \vdots & \ddots & \vdots  & \vdots\\
			\frac{\partial\varphi}{\partial x_d} & x_1\frac{\partial\varphi}{\partial x_d} & x_2\frac{\partial\varphi}{\partial x_d} & \cdots & \varphi +x_d\frac{\partial\varphi}{\partial x_d}  & *\\
			0 & 0 & 0 & \cdots & 0 & I_d
		\end{array}
		\right).
	\end{equation*}
	It is easy to show that $J\widetilde{\Psi}$ is of rank $2d$.
	Therefore, $\widetilde{\Psi}$ is a submersion.
	Lemma \ref{lem:Morse_basic2} implies that $Z=\{t \in F^{d+1}\;|\; \Psi_t \text{ has a degenerate critical point}\}$ is of dimension smaller than $d+1$.
	
	We next consider the map $\Pi:(F^{d+1} \setminus \mycl(Z))\times (U \times U \setminus \Delta_U) \to F \times F^d \times F^d$ given by $\Pi(t,x,y)=(\Psi(t,x)-\Psi(t,y), d\Psi_t(x), d\Psi_t(y))$, where $\Delta_U=\{(x,x) \in F^n \times F^n\;|\; x \in U\}$.
	The Jacobian $J\Pi$ of $\Pi$  is of rank $(2d+1)$ everywhere in the domain of $\Pi$ in the same manner as above.
	In particular, the origin $0$ is a regular value of $\Pi$.
	We can take a definable subset $Z'$ of $F^{d+1} \setminus \mycl(Z)$ of dimension $<d+1$ such that $0$ is a regular value of $\Pi_t$ for every  $t \in F^{d+1} \setminus (\mycl(Z) \cup Z')$ by Corollary \ref{cor:Sard}.
	Fix $t \in F^{d+1} \setminus (\mycl(Z) \cup Z')$.
	We have $0 \notin \Pi_t(U \times U \setminus \Delta_U)$ because the codomain of $\Pi_t$ is of dimension larger than the dimension of the domain of $\Pi_t$.
	This implies that $\Psi_t$ has only critical points with distinct critical values.
	
	We consider the map $\Theta:F^{d+1} \times U \to F \times F^d$ given by $\Theta(t,x)=(\Psi(t,x),d\Psi_t(x))$.
	The Jacobian $J\Theta$ is of rank $(d+1)$ everywhere.
	For every $z \in D$, we set $C_z=\{t \in F^{d+1}\;|\; (z,0) \text{ is a critical value of } \Theta_t\}$.
	We have $\dim C_z<d+1$ by Corollary \ref{cor:Sard}.
	We obtain $(z,0) \notin \Theta_t(U)$ for $t \in F^{d+1} \setminus C_z$.
	This means that $z$ is a regular value of $\Psi_t$. 
	Set $C=\bigcup_{z \in D}C_z$.
	We have $\dim C<d+1$ by Proposition \ref{prop:others}(3),(6).
	
	We have $\dim (\mycl(Z) \cup Z' \cup C)<d+1$ by Proposition \ref{prop:others}(4),(5).
	We can choose $(a_0,\ldots,a_d) \in F^{d+1} \setminus (\mycl(Z) \cup Z' \cup C)$ sufficiently close to the origin.
\end{proof}

\begin{lemma}\label{lem:basic_func}
	There exists a definable $\mathcal C^p$ function $P:F \to F$ satisfying the following conditions:
	\begin{itemize}
		\item $P(x)=0$ for every $x \in (-\infty,-1] \cup [1,\infty)$.
		\item $P(x)=1$ for every $x \in [-1/2,1/2]$.
		\item $P$ is nondecreasing on $(\infty,0]$ and nonincreasing on $[0,\infty)$.
	\end{itemize}
\end{lemma}
\begin{proof}
	The proof is a routine. In fact, we can take $P$ so that $P$ is semialgebraic.
	We omit the proof.
\end{proof}

Lemma \ref{lem:zero_appro} asserts that the existence of `very small' definable function.
In o-minimal structures, the existence of such functions follow from the more general approximation theorem by Escribano \cite{E}.
However, such an approximation theorem is not found in d-minimal setting.
We prepare  two lemmas for the proof of Lemma \ref{lem:zero_appro}.

\begin{lemma}\label{lem:bounded_below}
	Let $X$ be a bounded and closed definable set.
	Let $\psi:X \rightarrow F$ be a positive definable function which is locally bounded from below by positive constants; that is, for any $x \in X$, there exist an open subset $U_x$ of $X$ and $c_x>0$ such that $\psi>c_x$ on $U_x$.
	There exists a positive $c$ such that $\psi(x) > c$ for all $x \in X$.
\end{lemma}
\begin{proof}
	Assume for contradiction that the set $Z_t=\{x \in X\;|\; \psi(x) \leq t\}$ is not empty for each $t>0$.
	Set $C_t=\mycl(Z_t)$. 
	We can find $x_0 \in \bigcap_{t>0}C_t$ by \cite[Proposition 4.7]{Fuji-compact} because $X$ is bounded and closed.
	There exists an open neighborhood $U$ of $x_0$ in $X$ and $d>0$ such that $\psi>d$ on $U$ because $\psi$ is locally bounded from below by positive constants.
	Take $0<t<d$.
	We can find a point $x \in U$ such that $\psi(x) \leq t$ because $x_0 \in C_t = \mycl(Z_t)$.
	We have obtained a contradiction. 
\end{proof}

We denote the set of isolated points in $S$ by $\myIso(S)$ for any topological space $S$.
We set $\myLpt(S):=S \setminus \myIso(S)$.
Let $X$ be a nonempty closed subset of a topological space $S$.
We set $X\langle 0 \rangle=X$ and, for any $m>0$, we set $X \langle m \rangle = \myLpt(X \langle m-1 \rangle)$.
We say that $\myrank(X)=m$ if $X \langle m \rangle=\emptyset$ and $X\langle m-1 \rangle \neq \emptyset$.
We say that $\myrank X = \infty$ when $X \langle m \rangle \neq \emptyset$ for every natural number $m$.

\begin{lemma}\label{lem:bounded_below2}
	Suppose that $\mathcal F$ is d-minimal.
	Let $M$ be a $\mathcal D^p$ submanifold of $F^n$ and $\psi:M \to F$ is a positive definable function which is locally bounded from below by positive constants.
	Then there exists a positive $\mathcal D^p$ function $\varphi:M \to F$ such that $\varphi \leq \psi$ on $M$. 
\end{lemma}
\begin{proof}
	We may assume that $M$ is closed by Corollary \ref{cor:submanifold}.
	We first reduce to the case $M=F$.
	For any positive $u \in F$, we set $M_u=\{x \in M\;|\; \| x  \|^2 \leq u\}$.
	We can take $u_0>0$ so that $M_{u_0} \neq \emptyset$.
	Consider the map $\mu:[u_0,\infty) \rightarrow F$ given by $\mu(u)=\inf\{\psi(x)\;|\; x \in M_u\}$.
	The set $M_u$ is bounded and closed.
	The map $\mu$ is well-defined because $\mathcal F$ is definably complete.
	The map $\mu$ is definable and nonincreasing.
	It is positive by Lemma \ref{lem:bounded_below}.
	Set $\mu(u)=\mu(u_0)$ for every $u<u_0$.
	The definable map $\mu:F \to F$ is nonincreasing.
	In particular, it is locally bounded from below by positive constants.
	Let $\varphi_0:F \to F$ be a positive definable $\mathcal C^p$ function such that $\varphi_0<\mu$ on $F$.
	Consider the positive $\mathcal D^p$ map $\varphi:M \to F$ defined by $\varphi(x)=\varphi_0(\|x\|^2)$.
	We have $\varphi(x)=\varphi_0(\|x\|^2) < \mu(\|x\|^2) \leq \psi(x)$.
	We have succeeded in reducing to the case in which $M=F$.
	
	We consider the case in which $M=F$.
	Let $D$ be the closure of the set of points at which the function $\psi$ is not of class $\mathcal C^p$.
	It is a definable set and has a nonempty interior by Proposition \ref{prop:others}(2),(5).
	The Cantor-Bendixson rank $d=\myrank(D)$ is finite by \cite[Lemma 5.10]{Fornasiero_dmin}.
	We prove the lemma by induction on $d$.
	When $d=0$, the set $D$ is an empty set and we have nothing to prove in this case.
	
	We consider the case $d>0$.
	Set $E=\myIso(D)$.
	The definable set $E$ is discrete.
	Consider the definable map $\rho_0:E \to F$ defined by $\rho_0(x)=\inf \{|y-x|/2\;|\; y \in E, y \neq x \}$.
	We have $\rho_0(x)>0$ for every $x \in E$ because $E$ is discrete.
	We next choose positive definable functions $\rho_1,\tau_0: E \to F$ so that, for each $x \in E$, the inequality $\psi(y) \geq \tau_0(x)$ holds for every $y \in (x-\rho_1(x),x+\rho_1(x))$.
	We can choose such $\rho_1$ and $\tau_0$ by Proposition \ref{prop:others}(7) because $\psi$ is locally bounded from below by positive constants.
	We set $\rho(x)=\min\{\rho_0(x),\rho_1(x)\}$ and $I_x=(x-\rho(x),x+\rho(x))$ for $x \in E$.
	We have $I_x \cap I_{x'}=\emptyset$ whenever $x \neq x'$.
	The set $U=\bigcup_{x \in E} I_x$ is definable and open.
	Set $\tau(y)=\inf\{\psi(x)\;|\; x \in I_y\}$ for every $y \in E$.
	We obviously have $\tau(y) \geq \tau_0(y)$.
	In particular, we have $\tau(y) >0$ for every $y \in E$.
	Let $P:F \to F$ be the $\mathcal D^p$ function given in Lemma \ref{lem:basic_func}.
	We define definable functions $\theta, \sigma:F \to F$ as follows:
	\begin{align*}
		&\theta(x)=\left\{\begin{array}{ll}
			P\left(\dfrac{x-y}{\rho(y)}\right) & \text{ if } x \in I_y,\\
			0 & \text{ if } x \notin U
		\end{array}
		\right.\\
		&\sigma(x)=\left\{\begin{array}{ll}
			\tau(y) \cdot \theta(x) & \text{ if } x \in I_y,\\
			0 & \text{ if } x \notin U
		\end{array}
		\right.
	\end{align*}
	They are $\mathcal D^p$ functions off $D \setminus E$.
	We consider the map $g:F \to F$ given by $$g(x)=(1-\theta(x))\cdot \psi(x)/2+\sigma(x)/2.$$
	Let $D'$ be the closure of the set of points at which $g$ is not of class $\mathcal C^p$.
	The inclusion $D' \subseteq D \setminus E$ holds because $\theta$ and $\sigma$ is of class $\mathcal C^p$ off $D \setminus E$ and $\theta$ is identically one on a neighborhood of $E$.
	We have $$\myrank(D')<d.$$
	The function $\sigma$ is identically zero off $U$ and $0<\sigma \leq \psi$ on $U$.
	The function $\theta$ is zero off $U$ and $0 < \theta \leq 1$ on $U$.
	It is obvious that $0<g \leq \psi$ on $F$.
	
	We next prove that $g$ is locally bounded from below by positive constants.
	For every $x \in F$, we have only to find a positive $c$ such that $f(y) >c$ for every point $y$ in a neighborhood of $x$.
	It is obvious when $g$ is continuous at $x$.
	Therefore, we have only to consider the case in which $x \in D'$.
	We can take a positive constant $c'$ and an open interval $I$ such that $\psi >c'$ on $I$ because $\psi$ is locally bounded from below by positive constants.
	When $x \in D'$, we have $x \notin U$.
	When the endpoint of $I$ is contained in $I_y$ for some $y \in E$, we can shrink $I$ so that $I \cap I_y=\emptyset$ because $x \notin U$.
	By this procedure, we can reduce to the case in which either the equality $I \cap I_y = \emptyset$ or the inclusion $I_y \subseteq I$ holds for every $y \in E$.
	We have 
	\begin{align*}
		g(x) &\geq (1-\theta(x)) \cdot \inf\{\psi(z)\;|\; z \in I_y\}/2 + \tau(y) \cdot \theta(x)/2\\
		&=\tau(y)/2 \geq \inf\{\psi(z)\;|\;z \in I\}/2 >c'/2
	\end{align*}
	for every $x \in I_y \subseteq I$.
	We also have $g(x)=\psi(x)/2 >c'/2$ for every $x \in I \setminus U$.
	We have demonstrated that $g>c'/2$ on $I$ and it means that $g$ is locally bounded from below by positive constants.
	
	Apply the induction hypothesis to $g$.
	There exists a positive definable $\mathcal C^p$ function $\varphi:F \to F$ such that $\varphi \leq g$ on $F$.
	The definable function $\varphi$ is the desired definable $\mathcal C^p$ function. 
\end{proof}

We are now ready to prove Lemma \ref{lem:zero_appro}.

\begin{lemma}\label{lem:zero_appro}
	Suppose that $\mathcal F$ is d-minimal.
	Let $M$ be a $\mathcal D^p$ submanifold of $F^n$ and $\{D_i\}_{1 \leq i \leq k}$ be a finite family of $\mathcal D^{p-1}$ vector fields on $M$.
	There exists a positive $\mathcal D^p$ function $f:M \to F$ such that 
	$$\left|D_{\alpha} f\right|<\varepsilon$$
	on $M$ for every sequence $\alpha=(\alpha_1,\ldots,\alpha_m)$ of integers in $\{1,\ldots,k\}$ of length $m \leq p$, where $D_{\alpha}f:=D_{\alpha_1}\cdots D_{\alpha_m}f$.
\end{lemma}
\begin{proof}
	We may assume that $\varepsilon$ is of class $\mathcal C^p$ without loss of generality by applying Lemma \ref{lem:bounded_below2} to $\varepsilon$.
	
	Let $B$ be the open box in $F^n$ given by $B:=\{x=(x_1,\ldots,x_n) \in F^n\;|\; |x_i| <1 \text{ for each } 1 \leq i \leq n\}$.
	Let $\varphi:F^n \to B$ be the  $\mathcal D^p$ diffeomorphism given by $\varphi(x_1,\ldots,x_n)=(\frac{x_1}{\sqrt{1+x_1^2}}, \ldots, \frac{x_n}{\sqrt{1+x_n^2}})$.
	The inverse map $\psi:B \to F^n$ of $\varphi$ is given by $\psi(x_1,\ldots,x_n)=(\frac{x_1}{\sqrt{1-x_1^2}}, \ldots, \frac{x_n}{\sqrt{1-x_n^2}})$.
	Set $U=\varphi(M)$.
	
	The frontier $\partial U$ of $U$ in $F^n$ is bounded and closed.
	We can find a definable $\mathcal C^p$ function $h:F^n \to F$ whose zero set is $\partial U$ by Proposition \ref{prop:others}(1).
	We may assume that $0 \leq h <1$ by considering $h^2/(1+h^2)$ instead of $h$.
	We consider that the domain of definition of $h$ is $\mycl(U)$ by considering the restriction of $h$ to $\mycl(U)$  instead of $h$.
	
	We define $\mathcal I$ and $\mathcal J$ as follows:
	\begin{align*}
		\mathcal I&=\{\text{sequences of integers in }\{0,1,\ldots,k\} \text{ of length} \leq p\},\\
		\mathcal J&=\{\text{sequences of integers in }\{1,\ldots,k\} \text{ of length} \leq p\}.
	\end{align*}
	We denote the length of a sequence of $\alpha \in \mathcal I$ by $\mylen(\alpha)$. 
	For every $\alpha \in \mathcal J$ and $\beta \in \mathcal I$, we write $\beta \prec \alpha$ if $\mylen(\alpha)=\mylen(\beta)$ and either $\alpha_i=\beta_i$ or $\beta_i=0$, where $\alpha_i$ and $\beta_i$ are the $i$-th element of $\alpha$ and $\beta$, respectively.
	For $\alpha \in \mathcal J$ and $\beta \in \mathcal I$ with $\beta \prec \alpha$, $\alpha \setminus \beta$  denotes the unique element $\gamma$ in $\mathcal I$ such that $\gamma \prec \alpha$ and the $i$-th element of $\gamma$ is zero if and only if $\beta_i \neq 0$ for every $1 \leq i \leq \mylen(\alpha)$.
	We set $D_0f = f$ for every definable function $f$ on $M$ for simplicity of notations. 
	
	Set $\theta(x)=(\zeta(h(\varphi(x))))^{p+1}$ for every $x \in M$, where $\zeta:F \to F$ is a $\mathcal D^p$ function which will be determined later.
	By chain rule and Leibnitz rule, for every $\beta \in \mathcal I$, we can find definable continuous functions $\rho_{\beta,i_0,i_1,\ldots,i_p}:M \to F$ for $1 \leq i_0 \leq p+1$ and $0 \leq i_1,\ldots, i_p \leq p+1$ such that they are independent of the choice of $\zeta$ and the equality
	\begin{equation}
		D_{\beta}\theta(x)=\sum_{i_0=1}^{p+1}\sum_{i_1=0}^{p+1}\cdots\sum_{i_p=0}^{p+1} \rho_{\beta,i_0,i_1,\ldots,i_p}(x)(\zeta(h(\varphi(x))))^{i_0}\cdots (\zeta^{(p)}(h(\varphi(x))))^{i_p}\label{eq:aaa0}
	\end{equation}
	holds.
	Set $\mathcal G=\{\rho_{\beta,i_0,i_1,\ldots,i_p}(x)\;|\; \beta \in \mathcal I, 1 \leq i_0 \leq p+1, 0 \leq i_1,\ldots, i_p \leq p+1\}.$
	
	Consider the finite family of definable continuous functions $$\mathcal H=\displaystyle\left\{\dfrac{1}{\varepsilon(\psi(x))} \cdot \left(D_{\alpha}\varepsilon\circ \psi(x)\right)\cdot \eta(\psi(x))\;\middle|\; \alpha \in \mathcal I, \eta \in \mathcal G
	\right\}.$$
	Every element in $\mathcal H$ is a definable continuous function on $U$.
	In particular, it is locally bounded. 
	Apply Proposition \ref{prop:others}(8) to $h$ and $\mathcal H$.
	We can find a strictly increasing $\mathcal D^p$ function $\zeta:F \to F$ such that $\zeta(0)=0$, $\zeta(1)=1$, $p$-flat and 
	\begin{equation}
		\lim_{x \to y \in \partial U} \dfrac{\zeta(h(x))}{\varepsilon(\psi(x))} \cdot \left(D_{\alpha}\varepsilon\circ \psi(x)\right)\cdot \eta(\psi(x))=0 \label{eq:aaa1}
	\end{equation}
	for each $\alpha \in \mathcal I$ and $\eta \in \mathcal G$.
	Observe that the inequality $0 \leq \zeta(h(x)) <1$ holds for every $x \in U$.
	We define a positive $\mathcal D^p$ function $g:M \to F$ by $g(x)=\theta(x) \cdot \varepsilon(x)$.
	
	Fix an $\alpha \in \mathcal I$. 
	We have
	\begin{equation}
		D_{\alpha}g(x) =\sum_{\beta \prec \alpha}D_{\beta}\theta(x) \cdot D_{\alpha\setminus\beta}\varepsilon(x).\label{eq:aaa2}
	\end{equation}
	By equalities (\ref{eq:aaa0}), (\ref{eq:aaa1}) and (\ref{eq:aaa2}) and the condition that $\zeta$ is $p$-flat, we have
	\begin{equation}
		\lim_{x \to y \in \partial U} \dfrac{1}{\varepsilon(\psi(x))} \cdot (D_{\alpha}g) \circ \psi(x) =0. \label{eq:aaa3}
	\end{equation} 
	
	Set $C=\left\{x \in U\;\middle|\;\left|(D_{\alpha}g) \circ (\psi(x))\right| \geq \varepsilon(\psi(x)) \text{ for some }\alpha \in \mathcal I \right\}.$
	The definable set $C$ is closed in $U$ and $C$ does not intersect with a neighborhood of $\partial U$ in $F^n$ by equality (\ref{eq:aaa3}).
	It implies that $C$ is closed in $F^n$.
	Observe that $\dfrac{(D_{\alpha}g) \circ (\psi(x))}{\varepsilon(\psi(x))}  $ is continuous on $C$ for every $\alpha \in \mathcal I$.
	We can find $N>0$ such that $\left|\dfrac{(D_{\alpha}g) \circ (\psi(x))}{\varepsilon(\psi(x))}   \right|<N$ on $C$ by \cite[Corollary(Max-min theorem)]{M}.
	Let $h:M \to F$ be the positive $\mathcal D^p$ function given by $h(x)=g(x)/(2N)$.
	We have $\left|D_{\alpha} h\right|<\varepsilon \circ \psi|_U$ on $U$.
	The function $f$ defined by $f=h \circ \varphi|_M$ is a desired function.
\end{proof}


The following is our main theorem.
The $\mathcal D^p$ topology in the theorem is defined in \cite{E}. 
\begin{theorem}\label{thm:Morse}
	Let $p$ be a positive integer larger than one.
	Suppose that $\mathcal F$ is d-minimal.
	Let $M$ be a $\mathcal D^p$ submanifold of $F^n$.
	The set of all $\mathcal D^p$ Morse functions on $M$ is dense in $\mathcal D^p(M)$ equipped with the $\mathcal D^p$ topology.
\end{theorem}
\begin{proof}
	We take arbitrary $h \in \mathcal D^p(M)$ and arbitrary positive definable continuous function $\varepsilon$ defined on $M$.
	Consider the set
	$$\mathcal W_{h,\varepsilon}=\{g \in \mathcal D^p(M)\;|\; |D_{i_1}\cdots D_{i_j}(g-h)|< \varepsilon \text{ for } 1 \leq i_1,\ldots, i_j \leq n, j \leq p\}.$$
	We have only to construct a $\mathcal D^p$ Morse function $g$ in $\mathcal W_{h,\varepsilon}$.
	
	Set $d=\dim M$.
	The $\mathcal D^p$ submanifold  $M$ is covered by finitely many $\mathcal D^p$ multi-valued graphs $U_1,\ldots, U_k$ by Lemma \ref{lem:multi_graph_cover}.
	Let $U_i$ be a $\mathcal D^p$ multi-valued graph with respect to a coordinate projection $\pi_i:F^n \to F^d$ for $1 \leq i \leq k$.
	We can take a $\mathcal D^p$ function $\lambda_i:M \to F$ with $\lambda_i^{-1}(0)=M \setminus U_i$ by Proposition \ref{prop:others}(1).
	Set $\varepsilon'=\varepsilon/k$ and $g_0=h$.
	We also put $V_i=\bigcup_{j=1}^i U_j$ and $V_0=\emptyset$.
	
	By induction on $i$, we have only to construct a $\mathcal D^p$ function $g_i:M \to F$ such that $g_i \in \mathcal W_{g_{i-1},\varepsilon'}$ and  the restriction $g_i|_{V_i}$ is a $\mathcal D^p$ Morse function.
	We may assume that $\pi_i$ is the coordinate projection of $F^n$ onto the first $d$ coordinates by permuting the coordinates if necessary.
	Let $D_i$ be the set of all critical values of the restriction $g_{i-1}|_{V_{i-1}}$.
	We have $\dim D_i \leq 0$ by Corollary \ref{cor:nondegenerate}.
	We consider the definable function $g_i:M \to F$ given by  $$g_i(x_1,\ldots,x_n) =g_{i-1}(x_1,\ldots,x_n)+\lambda_i(x)(a_0\varphi(x)+\sum_{j=1}^d a_jx_j\varphi(x)),$$ where $\varphi$ and $a_0, \ldots, a_d$ will be determined later.
	
	We set $D_0f=f$ for every definable function $f$ on $M$.
	We also set $x_0=1$ for simplicity of notations.
	We define $\mathcal I$ and $\mathcal J$, and use the notations such as $\beta \prec \alpha$ in the same manner as the proof of Lemma \ref{lem:zero_appro}.
	We want to choose $\varphi \in \mathcal D^p(M)$ so that the inequality $\left|D_{\alpha} (g_i-g_{i-1})\right|<\varepsilon$ holds for every $\alpha \in \mathcal J$ whenever we take $a_0,a_1,\ldots,a_d$ with $|a_i|< 1/(2^p(d+1))$ for each $0 \leq i \leq d$.
	Set $\mathcal H:=\{D_{\alpha}(\lambda x_j)\;|\; 0 \leq j \leq d,\ \alpha \in \mathcal I\} $. 
	Let $u:M \to F$ be the definable continuous function defined by $u(x)=\max\{1,\max\{h(x)\;|\; h \in \mathcal H\}\}.$
	Using Lemma \ref{lem:zero_appro}, we can take $\varphi \in \mathcal D^p(M)$ so that the inequality $|D_{\alpha}\varphi|< \varepsilon/u$ holds on $M$ for each $\alpha \in \mathcal I$.
	For every $\alpha \in \mathcal J$, we have
	\begin{align*}
		\left|D_{\alpha} (g_i-g_{i-1})\right| &\leq \sum_{j=0}^d\sum_{\beta \prec \alpha}|a_j| \cdot |D_{\beta}(\lambda x_j)| \cdot |D_{\alpha\setminus\beta}\varphi|\\
		& < \sum_{j=0}^d\sum_{\beta \prec \alpha}(1/(2^p(d+1))) \cdot u \cdot \varepsilon/u = \varepsilon.
	\end{align*}
	We have successfully chosen $\varphi$ so that the inequality $\left|D_{\alpha} (\Phi-f)\right|<\varepsilon$ holds.
	
	We next choose $a_0,a_1,\ldots,a_d \in F$. 
	We can take $a_0,a_1,\ldots,a_d$ so that $|a_i|< 1/(2^p(d+1))$ for each $0 \leq i \leq d$ and $g_i|_{U_i}$ is a Morse function and every element in $D_i$ is a regular value of $g_i|_{U_i}$ by Lemma \ref{lem:Morse_basic}.
	Since $g_i$ coincides with $g_{i-1}$ off $U_i$, the restriction $g_i|_{V_i}$ is a $\mathcal D^p$ Morse function. 
\end{proof}

The set of all $\mathcal D^p$ Morse functions on $F$ is not necessarily open in $\mathcal D^p(F)$.
The structure $\mathcal F=(\mathbb R,<,+,\cdot,2^{\mathbb Z})$ is d-minimal \cite{D2}.
We construct a definable More function $f:\mathbb R \to \mathbb R$ whose arbitrary neighborhood in $\mathcal D^2(\mathbb R)$ contains a non-Morse function.

Let $\eta_\delta:\mathbb R \to \mathbb R$ be a semialgebraic $\mathcal C^2$ function such that 
\begin{itemize}
	\item $\eta_\delta$ is identically zero on $(-\infty,-2\delta) \cup (2\delta,\infty)$;
	\item $\eta_\delta$ is identically one on $(-\delta,\delta)$;
	\item $\eta_\delta$ is strictly increasing on $[-2\delta,-\delta]$;
	\item $\eta_\delta$ is strictly decreasing on $[\delta,2\delta]$,
\end{itemize}
where $\delta$ is a positive real number.
Let $\varphi:\mathbb R \to \mathbb R$ be the definable function given by $\varphi(x)=(1-x^2)\eta_\delta(x)+(1/2-(x-1)^2/8)\eta_\delta(x)+Nx^3(x-1)^3$.
Fix a sufficiently small $\delta>0$ and a sufficiently large $N>0$.
We can easily show that the restriction of $\varphi$ to a small neighborhood of $[0,1]$ is a Morse function having critical points only at $x=0,1/2,1$ and $\varphi(0)=1,\varphi''(0)=-1,\varphi(1/2)<0, \varphi(1)=1/2,\varphi''(1)=-1/8$.

Consider the function $\psi$ defined on a small neighborhood $[2,\infty)$ given by $\psi(x)=2^{-n}\cdot \varphi(2^{-n}(x-2^n))$ on $[2^n,2^{n+1}]$ for every $n \in \mathbb N$.
This function is a definable Morse function and has critical points only at $x=2^n$ and $x=3 \cdot 2^{n-1}$ for $n \in \mathbb N$.
Using this, we can construct a definable Morse function $f:\mathbb R \to \mathbb R$ such that $f$ has critical points only at $x=0$, $x=2^n$ and $x=3 \cdot 2^{n-1}$, and $f(0)=0$, $f(2^n)=2^n$ and $f(3 \cdot 2^{n-1})<0$ for $n \in \mathbb N$.
Take an arbitrary positive definable continuous function $\varepsilon:\mathbb R \to \mathbb R$.
Let $g:\mathbb R \to \mathbb R$ be a nonnegative semialgebraic $\mathcal C^2$ function such that $g(0)=1$, $g^{-1}(0) \subseteq (-\infty,-1] \cup [1,\infty)$, $g''(0) \neq 0$ and $g'$ vanishes only at $x=0$ in $(-1,1)$.
We can take $n \in \mathbb N$ so that $|2^{-n}g^{(p)}|<\varepsilon$ for $0 \leq p \leq 2$ because $g$ is identically zero outside a compact set.
The function $f+2^{-n}g$ is contained in the $\varepsilon$-neighborhood of $f$ in $\mathcal D^2(\mathbb R)$, but it has a common critical value at $x=0$ and $x=2^{n}$.
%

\end{document}